\documentclass[12pt]{amsart}
\usepackage{tikz-cd}
\usepackage[outline]{contour}
\contourlength{0.8pt} 

\usepackage{verbatim}
\usepackage{xcolor}
\usepackage{url}
\usepackage{hyperref}
\usepackage{amssymb}
\usepackage{enumitem}
\usepackage[margin=3.3cm]{geometry}
\usepackage{mathtools}
\usepackage{graphicx}
\usepackage{setspace}
\numberwithin{equation}{section}
\numberwithin{figure}{section}
\usetikzlibrary{calc,intersections}
\usepackage{changepage} 

\let\cal\mathcal
\def\Ascr{{\cal A}}

\def\Dscr{{\cal D}}
\def\Escr{{\cal E}}
\def\Fscr{{\cal F}}
\def\Gscr{{\cal G}}
\def\Hscr{{\cal H}}
\def\Iscr{{\cal I}}

\def\Oscr{{\cal O}}
\def\Pscr{{\cal P}}

\def\Sscr{{\cal S}}
\def\Tscr{{\cal T}}

\def\Zscr{{\cal Z}}

\let\blb\mathbb
\def\CC{{\blb C}} 
 
\def\EE{{\blb E}} 
 
\def\FF{{\blb F}} 
\def\QQ{{\blb Q}}
\def\GG{{\blb G}}
\def \PP{{\blb P}}

\def \ZZ{{\blb Z}}

\def \RR{{\blb R}}
\def \HH{{\blb H}}

\def \aa{{\mathbf a}}
\def \bb{{\mathbf b}}
\def \ppp{{\mathbf p}}
\def \qqq{{\mathbf q}}

\let\oldmarginpar\marginpar
\def\marginpar#1{\setlength{\marginparwidth}{0.15\textwidth}\oldmarginpar{\tiny \baselineskip 0pt\lineskip 0pt\lineskiplimit 0pt  \raggedright #1}}

\def\id{\text{id}}

\def\Gr{\operatorname{Gr}}

\def\Gr{\operatorname{Gr}}

\def\length{\mathop{\text{length}}}

\def\ch{\mathop{\text{Ch}}}

\def\coh{\mathop{\text{\upshape{coh}}}}

\def\Spec{\operatorname {Spec}}
\def\Vect{\operatorname {Vect}}

\def\GL{\operatorname {GL}}

\def\PGL{\operatorname {PGL}}

\def\Ext{\operatorname {Ext}}

\def\Hom{\operatorname {Hom}}
\def\Sch{\operatorname {Sch}}

\def\RHom{\operatorname {RHom}}

\def\H{\operatorname {H}}

\def\Gal{\operatorname {Gal}}
\def\id{{\operatorname {id}}}

\def\Pic{\operatorname {Pic}}
\def\Nef{\operatorname {Nef}}

\DeclareMathOperator{\Alg}{Alg}

\DeclareMathOperator{\Aut}{Aut}

\DeclareMathOperator{\Bl}{Bl}

\def\K{\operatorname{K}}
\def\R{\operatorname{R}}
\def\RHom{\operatorname{RHom}}
\def\length{\operatorname{length}}
\def\td{\operatorname{td}}

\newtheorem{lemma}{Lemma}[section]

\newtheorem{theorem}[lemma]{Theorem}
\newtheorem{corollary}[lemma]{Corollary}

\theoremstyle{definition}

{

}

\theoremstyle{remark}

\newtheorem{remark}[lemma]{Remark}

\newdimen\uboxsep \uboxsep=1ex
\def\uboxn#1{\vtop to 0pt{\hrule height 0pt depth 0pt\vskip\uboxsep
\hbox to 0pt{\hss #1\hss}\vss}}

\def\uboxs#1{\vbox to 0pt{\vss\hbox to 0pt{\hss #1\hss}
\vskip\uboxsep\hrule height 0pt depth 0pt}}

\newtheoremstyle{named}{}{}{\itshape}{}{\bfseries}{.}{.5em}{\thmnote{#3's }#1}
\theoremstyle{named}

\def\Stab{\operatorname{Stab}}

\makeatletter
\@namedef{subjclassname@2020}{\textup{2020} Mathematics Subject Classification}
\makeatother

\title[Disconnected Stab(X) and infinitely many braid group orbits]{Disconnected Stab(X) and infinitely many braid group orbits} 

\subjclass[2020]{}
\author{Anya Nordskova}
\thanks{}
\email{anya.nordskova@ipmu.jp}
\address[Anya Nordskova]{Kavli IPMU (WPI), UTIAS, The University of Tokyo, Kashiwa, Chiba 277-8583, Japan}
\begin{document}
\begin{abstract} We show that the space of numerical locally finite Bridgeland stability conditions $\Stab(Y)$ on a very general quintic threefold $Y$ has infinitely many connected components.  Using similar methods,  we also prove that for a particular smooth projective (weak Fano) threefold $X$ the braid group action on the set of full exceptional collections in $\Dscr^b(X)$ has infinitely many orbits,  in particular,  it is not transitive.  This provides a counterexample to a conjecture of Bondal and Polishchuk from 1993. The conjecture was first disproved by Chang,  Haiden,  and Schroll,  who constructed a family of partially wrapped Fukaya categories for which the transitivity fails.  However,  no counterexample of the form $\Dscr^b(X)$ for $X$ a smooth projective variety was previously known.  In addition,  we show that $\Stab(X)$ also has infinitely many connected components.  These are also the first known examples of smooth projective varieties whose spaces of Bridgeland stability conditions are disconnected.  
\end{abstract}
\maketitle
 \tableofcontents 
\section{Introduction}
\subsection{Quintic threefolds and disconnected Stab}
In \S \ref{sec:quintic} we show that the space of numerical locally finite Bridgeland stability conditions $\Stab(Y)$ has infinitely many connected components for two kinds of quintic threefolds $Y$.  Namely,  first we consider symmetric quintic threefolds $Y$ possessing an involution $\Phi$ which interchanges
two disjoint $(-1,-1)$-lines $\ell_1,\ell_2$,  see \S \ref{subsec:symmetricquintic}.  Then we prove the same claim for a very general quintic threefold $Y$, see \S \ref{subsec:verygeneralquintic}.  The argument is based on the following simple observation: an autoequivalence of $\Dscr^b(Y)$ acting trivially on $K_{num}(Y)$
and fixing a stability condition automatically fixes its entire connected component pointwise (Lemma~\ref{lem:strategy_quintic}).
For a symmetric quintic,  the involution $\Phi$ fixes every geometric stability condition $\sigma$ in the family constructed by Li \cite[Theorem 1.3]{LiQuintic}, but does not fix their images under (non-zero) powers of the spherical twist $T_{\Oscr_{\ell_1}(-1)}$.  From this we conclude that all stability conditions $T_{\Oscr_{\ell_1}(-1)}^n(\sigma)$ with $n \in \ZZ$ lie in pairwise distinct connected components (Theorem \ref{thm:quintic}). 

For a very general quintic $Y$ we replace the involution $\Phi$ by
semi-linear equivalences of $\Dscr^b(Y)$ induced by automorphisms of $\CC$, which fix the coefficient field of the defining equation of $Y$.  By the work of Sosna \cite{Sosna} semi-linear equivalences act on $\Stab(Y)$ in the same way as usual $\CC$-linear autoequivalences.  A general quintic threefold has $2875$ lines,  which are all pairwise disjoint $(-1,-1)$-curves.  Field automorphisms as above act transitively on lines in $Y$.  Moreover,  the semi-linear equivalences they induce fix Li's stability conditions and act trivially on $K_{num}(Y)$.  Since the lines are disjoint $(-1,-1)$-curves,
the same argument as for the symmetric quintic applies to every two distinct lines $\ell_1, \ell_2 \subset Y$, i.e. for every line $\ell_1$
the stability conditions $T_{\Oscr_{\ell_1}(-1)}^n(\sigma)$,
$n\in\ZZ$,  are in pairwise distinct connected components
(Theorem \ref{thm:very_general_quintic}).

In Remark \ref{rem:disconnected_aut} we further use the symmetric
quintic to construct a smooth Fano fivefold
$Z=\Bl_Y\PP^5$ for which
$\Stab(Z)/\Aut\Dscr^b(Z)$ is disconnected, i.e.  there are distinct connected components of $\Stab(Z)$ not related by an autoequivalence. 

Although no smooth projective variety with disconnected Bridgeland stability space was previously known,  it is worth mentioning that disconnected $\Stab$ for certain Verdier quotients of $\Dscr^b(X)$ with $X$ smooth projective had appeared in the work of Meinhardt and Partsch,  see \cite[Theorem 1.3]{MeinhardtPartsch}.

\subsection{Infinitely many braid group orbits}
Now let $\Dscr$ be a $\Hom$-finite triangulated category with an exceptional collection of length $n$.  As discovered by Bondal \cite[Assertion~2.3]{Bondal89},  the set of all full exceptional collections in $\Dscr$ carries a natural action of the $n$-strand braid group $B_n$ generated by mutations. 

In \cite[Conjecture 2.2]{BondalPolishchuk} Bondal and Polishchuk conjectured that this action is always transitive,  up to independent shifts of the $n$ exceptional objects.  Much later the conjecture was disproved by Chang, Haiden and Schroll \cite[Theorem~5.4]{ChangHaidenSchroll},  who constructed a series of partially wrapped
Fukaya categories of graded surfaces (equivalently,
derived categories of graded gentle algebras) with infinitely many braid group orbits.  Subsequently,  Nakago and Takahashi
\cite{NakagoTakahashi} gave a purely algebraic proof of the same non-transitivity result.  

Now restrict to the case when $\Dscr$ is the derived category $\Dscr^b(X)$ of a smooth projective variety $X$.  In this setting the transitivity has been established for del Pezzo surfaces \cite{KuleshovOrlov},  the second Hirzebruch surface $\FF_2$ \cite{IshiiOkawaUehara},  and Fano threefolds of Picard rank $1$ possessing a full exceptional collection 
\cite{NVdB}.  Chang-Haiden-Schroll ask \cite[Question 5.6]{ChangHaidenSchroll} a natural question whether a counterexample of the form $\Dscr^b(X)$ with $X$ a smooth projective variety can occur.  We answer this question affirmatively, following precisely the strategy suggested in \cite{ChangHaidenSchroll}, namely,  to distinguish braid orbits by the action of an autoequivalence.  More precisely,  the authors observe (see Lemma \ref{lem:CHS} below) that an autoequivalence which preserves one full exceptional collection term-wise,  up to shifts,  automatically also preserves every exceptional collection in its braid group orbit.  Hence,  to produce a counterexample it suffices to find another full exceptional collection which is \emph{not} preserved by the same autoequivalence up to shifts.

Our construction is roughly as follows.  We choose a smooth rational sextic curve $C \subset \PP^3$ contained in a cubic surface $S \subset \PP^3$ and an involution $\varphi \in \PGL(4,\CC)$ such that $\varphi$ preserves $C$ and interchanges two lines $\ell_1, \ell_2 \in \PP^3$ with $\length(\ell_i \cap C) = 4$.   The variety producing a counterexample is $X=\Bl_C\PP^3$.  Orlov's blow-up formula gives a $6$-term full
exceptional collection $\EE$ on $X$, glued from the standard full exceptional collections on $\PP^3$ and $C \cong \PP^1$. We show that the lift $\Phi$ of $\varphi$ to $X$ fixes $\EE$ term-wise.  On the other hand,  we show that the strict transforms $\widetilde{\ell_i}$ of our lines $\ell_i$ are $(-1,-1)$-curves on $X$,  and hence $\Sscr_i := \Oscr_{\widetilde{\ell_i}}(-1)$ are spherical in $\Dscr^b(X)$.  Moreover, $\Phi(\Sscr_1) = \Sscr_2$.  Using the standard relation $\Phi \circ T_{\Sscr_1} \cong T_{\Phi(\Sscr_1)} \circ \Phi$ for spherical twists we then show that $T_{\Sscr_1}(\EE)$ is not preserved by $\Phi$.  Thus,  as explained above,  $\EE$ and $T_{\Sscr_1}(\EE)$ are not in the same braid group orbit,  see Theorem \ref{thm:twoorb}. 

We then show that in fact the braid group action on the set of full exceptional collections in $\Dscr^b(X)$ has infinitely many orbits (Theorem \ref{thm:main}, Corollary \ref{cor:inf}).  Moreover,  in Corollary \ref{cor:K0_invisible} we present an infinite family of orbits which are indistinguishable on the level of the Grothendieck group.  

In addition,  in Theorem \ref{thm:stab} we show that the space of Bridgeland stability conditions $\Stab(X)$ on $X$ has infinitely many connected components,  even over the same central charge.  The main ingredient is Lemma \ref{lem:stab_components},  which allows us to distinguish connected components containing algebraic stability conditions associated to two exceptional collections via the action of an autoequivalence,  in the same way as with braid group orbits.  

\subsection{Acknowledgement} The author is very grateful to Fabian Haiden for inspiring discussions during the conference ``Symplectic topology meets representation theory'',  which took place in Cambridge in April 2026,  as well as to the organisers of that event.  The author would also like to thank Yukinobu Toda for drawing her attention to the question of (dis)connectedness of $\Stab(X)$ and for suggesting to consider quintic threefolds,  as well as to Tom Bridgeland, Alexey Bondal and Alexander Kuznetsov for interesting and encouraging discussions.  This work was supported by World Premier International Research Center Initiative (WPI),  MEXT,  Japan.

\subsection{AI disclosure} We would like to acknowledge the assistance of OpenAI's ChatGPT 5.6 in finding numerous mistakes in earlier attempted constructions of the weak Fano threefold $X$,  leading to subsequent successful modifications.  It has also found the reference \cite{BlancLamy}.  It has not,  however,  produced the actual counterexample or any of the statements or proofs,  nor has it been used to generate any part of this text. 

\subsection{Conventions and notation.}
Unless explicitly said otherwise,  all varieties,  categories and functors are defined over $\CC$. For a smooth projective variety $X$ we write $
\Dscr^b(X):=\Dscr^b(\coh(X))$ for the bounded derived category of coherent sheaves on $X$ and $\K_0(X)$ for its Grothendieck group.  By $\Stab(X)$ we denote the space of locally finite numerical Bridgeland stability conditions on $\Dscr^b(X)$,  see  \cite{BridgelandStab} for details.  All statements also hold for Bridgeland stability conditions satisfying the support property.   Note that $K_0(X)$ coincides with the numerical Grothendieck group $K_{num}(X)$ when $X$ has a full exceptional collection.  
Write 
\[
\Zscr : \Stab(X)\to\Hom_{\ZZ}(K_0(X),\CC)
\]
for the central charge map.
We write $\RHom_X(A,B)$ for the derived Hom complex and
$\Ext_X^k(A,B)=\H^k\RHom_X(A,B)$.  For a spherical object $\Sscr\in\Dscr^b(X)$ we denote by
$T_{\Sscr} \in \Aut \Dscr^b(X)$ the corresponding spherical twist functor \cite{SeidelThomas}, where $\Aut \Dscr^b(X)$ stands for the group of isomorphism classes of exact autoequivalences of $\Dscr^b(X)$.  For an automorphism $\Phi \in \Aut X$ of a variety $X$ by a slight abuse of notation we will denote the corresponding autoequivalence by the same letter $\Phi \in \Aut \Dscr^b(X)$. 
 We denote by $B_n$ the braid group on $n$ strands. The group $B_n\ltimes\ZZ^n$ acts on full exceptional collections of length $n$, where $B_n$ acts by mutations and $\ZZ^n$ by independent shifts,  see \cite{Bondal89} for details.  

\section{Disconnectedness of $\Stab$ for quintic threefolds}\label{sec:quintic}
Below in \S \ref{subsec:symmetricquintic} we prove that the space of Bridgeland stability conditions $\Stab(Y)$ has infinitely many connected components for a \emph{symmetric} quintic threefold $Y$ possessing an involution which interchanges two $(-1,-1)$-lines.  Then in \S \ref{subsec:verygeneralquintic} we adapt the method to obtain the same conclusion for very general quintic threefolds. 
\subsection{A symmetric quintic threefold}\label{subsec:symmetricquintic}
The construction is based on the following general observation. 

\begin{lemma}\label{lem:strategy_quintic} Let $Y$ be a smooth projective variety of dimension $d \geq 2$ and let $\Phi \in \Aut \Dscr^b(Y)$ be an autoequivalence acting as the identity on $K_{num}(Y)$.  Moreover,  assume that there exist (i) two orthogonal spherical objects $\Sscr_1, \Sscr_2 \in \Dscr^b(Y)$ such that $\Phi(\Sscr_1) \cong \Sscr_2$ and (ii) a stability condition $\sigma \in \Stab(Y)$ such that $\Phi(\sigma) = \sigma$.
Then \begin{enumerate}
\item $\Phi$ fixes the entire connected component $\Sigma \subset \Stab(Y)$ of $\sigma$ pointwise, 
\item The components $T_{\Sscr_1}^n(\Sigma)$ with $n \in \ZZ$ are pairwise distinct. 
\end{enumerate} \end{lemma} 

\begin{proof} (1) Since $\Phi$ fixes $\sigma$,  it preserves $\Sigma$.  Put $F:=\{\rho\in\Sigma \mid \Phi(\rho)=\rho\}$.  The set $F$ is non-empty and closed.  Fix some 
$\rho\in F$.  Choose an open neighbourhood $U_0$ of $\rho$ such that $\Zscr|_{U_0}$ is injective and put $U:=U_0\cap {\Phi}^{-1}(U_0)$.  For every $\tau\in U$ one has $\tau \in U_0$ and $\Phi(\tau) \in U_0$.  Since $\Phi$ acts trivially on $K_{num}(Y)$,
\[
\Zscr(\Phi(\tau))=\Zscr(\tau).
\]
Then $\Phi(\tau) = \tau$ by the injectivity of $\Zscr|_{U_0}$. 
So $U\subset F$,  hence $F$ is open.  Since $\Sigma$ is connected,  $F=\Sigma$.

(2) Write $T_1$ for $T_{\Sscr_1}$.  It is sufficient to prove that
$T_1^n(\Sigma) \neq \Sigma$ for all $0 \neq n\in\ZZ$.
Suppose to the contrary that $T_1^n(\Sigma)=\Sigma$ for some non-zero $n\in\ZZ$.  Then by (1) the stability condition
$\sigma_1:=T_1^n(\sigma)$ is fixed by $\Phi$.

For a non-zero object $E$ denote by $\phi_\rho^+(E)$ the maximal phase of the Harder--Narasimhan factors of $E$ with respect to a stability condition $\rho$.  Let $\phi:=\phi_\sigma^+(\Sscr_1)$.  Since $\Phi(\sigma)=\sigma$ and $\Phi(\Sscr_1)\cong\Sscr_2$,  one has $\phi_\sigma^+(\Sscr_2)=\phi$.
Since $\Sscr_1$ and $\Sscr_2$ are orthogonal,
$T_1^{-n}(\Sscr_2)\cong\Sscr_2$, so
$\phi_{\sigma_1}^+(\Sscr_2)=\phi$.
On the other hand,
$T_1^{-n}(\Sscr_1)\cong\Sscr_1[n(d-1)]$, so
\[
\phi_{\sigma_1}^+(\Sscr_1)
=\phi_\sigma^+(\Sscr_1[n(d-1)])
=\phi+n(d-1).
\]
But $\Phi(\sigma_1)=\sigma_1$ and $\Phi(\Sscr_1)\cong\Sscr_2$ imply $\phi_{\sigma_1}^+(\Sscr_1)=\phi_{\sigma_1}^+(\Sscr_2)$, which contradicts $n\neq0$ and $d\geq2$.
\end{proof} 

\begin{remark} In fact, the assumptions of Lemma \ref{lem:strategy_quintic} force the dimension $d$ to be odd. Indeed,  since $\Phi$ acts trivially on $K_{num}(Y)$,  the two orthogonal spherical objects $\Sscr_1$ and $\Sscr_2$ have the same classes in $K_{num}(Y)$.  Hence $0 = \chi(\Sscr_1,\Sscr_2) = \chi(\Sscr_1,\Sscr_1) = 1 + (-1)^d$.  Below we apply Lemma \ref{lem:strategy_quintic} with $d =3$. \end{remark} 

Now we take $Y \subset \PP^4$ to be a smooth quintic threefold containing two lines $\ell_1,\ell_2$ which are $(-1,-1)$-curves exchanged by an involution $\Phi \in \Aut(Y)$.  For instance,  for $\lambda \in \CC$ let $Y_\lambda := V(F_\lambda) \subset \PP^4$, where
\[
F_\lambda := x_0^4x_1+x_1^4x_0+x_2^4x_3+x_3^4x_2+x_4^5 +\lambda x_4(x_0^2x_2^2+x_1^2x_3^2),
\]
and consider a projective involution
\[
\Phi[x_0:x_1:x_2:x_3:x_4]
=
[x_1:x_0:x_3:x_2:x_4]
\]
preserving $Y_\lambda$.  Every $Y_\lambda$
contains the two disjoint lines
\[
\ell_1:= \{x_1 = x_3 = x_4=0 \},
\quad
\ell_2:= \{x_0 = x_2 = x_4= 0\}
\]
satisfying $\Phi(\ell_1)=\ell_2$.  A direct computation shows that $Y_0$ is smooth.  Hence there is a Zariski open neighbourhood
$U \ni 0$ such that $Y_\lambda$ is smooth for every
$\lambda\in U$.  We can take $Y := Y_\lambda$ for $\lambda \in U \cap \CC^*$. 

Recall the following well-known fact: 

\begin{lemma}[{\cite[Lemma 3.4, Examples~3.20]{SeidelThomas}}]\label{lem:general_sph}
Let $X$ be a smooth projective threefold and let
$\iota: \ell\hookrightarrow X$ be a $(-1,-1)$-curve,  i.e.  a smooth rational curve such that
\[
N_{\ell/X}\cong
\Oscr_{\PP^1}(-1)\oplus\Oscr_{\PP^1}(-1).
\]
Then $K_X \cdot \ell = 0$ and for every line bundle $L$ on $\ell$,  the object
$\iota_*L\in\Dscr^b(X)$ is $3$-spherical.
\end{lemma}

\begin{lemma}\label{lem:m1m1quintic}
The lines $\ell_i$, $i=1,2$ in $Y$ are $(-1,-1)$-curves. 
\end{lemma}
\begin{proof}
The homogeneous coordinates on $\ell_1$ are $[x_0:x_2]$ and the three normal directions to $\ell_1$ in $\PP^4$ are represented by $x_1,x_3,x_4$. 
The normal bundle sequence for the inclusions
$\ell_1\subset Y\subset\PP^4$ is
\[
0 \to N_{\ell_1/Y}
\to \Oscr_{\PP^1}(1)^{\oplus3}
\overset{dF_\lambda}{\to}\Oscr_{\PP^1}(5)
 \to 0.
\]
Restricting the partial derivatives of
$F_\lambda$ to $\ell_1$ induces a map $
H^0(\Oscr_{\PP^1}(1))^{\oplus3} \to H^0(\Oscr_{\PP^1}(5))$
sending $(a,b,c)$ to $(x_0^4 a + x_2^4 b + \lambda x_0^2 x_2^2 c)$.  Since $\lambda \neq 0$,  this map is an isomorphism. Hence in particular $\H^0(\ell_1,N_{\ell_1/Y})=0$. 
On the other hand,
\[
\deg N_{\ell_1/Y}
=
\deg N_{\ell_1/\PP^4}-\deg N_{Y/\PP^4}|_{\ell_1}
=3-5=-2.
\]
So we have
$N_{\ell_1/Y}\cong \Oscr_{\PP^1}(a)\oplus \Oscr_{\PP^1}(b)$ with $a+b=-2$, while the vanishing of global sections gives $a<0$ and $b<0$.  Thus $a=b=-1$. The argument for $\ell_2$ is identical. 
\end{proof}

For $i=1,2$, set $\Sscr_i:=\Oscr_{\ell_i}(-1)$. 
By Lemma \ref{lem:m1m1quintic} and Lemma \ref{lem:general_sph} the sheaves $\Sscr_i$ are spherical.  Since the lines $\ell_1$ and $\ell_2$ are disjoint,  $\Sscr_1$ and $\Sscr_2$ are orthogonal.  
Moreover,  one has $\Phi(\Sscr_1)\cong\Sscr_2$,  because $\Phi$ is induced by an involution exchanging the two lines $\ell_1$ and $\ell_2$.  Since $Y$ has Picard rank $1$, any automorphism of $Y$ fixes the ample generator $H$ and therefore acts trivially on $K_{num}(Y)$,  because the numerical even cohomology is generated by $1, H, H^2$ and $H^3$. 

 In \cite[Theorem 1.3]{LiQuintic} Li constructs 
a continuous family 
\[
\sigma_{\alpha,\beta,H}^{a,b}
=
(Z_{\alpha,\beta,H}^{a,b},\Ascr^{\alpha,\beta,H}(Y))
\]
of geometric Bridgeland
stability conditions on $\Dscr^b(Y)$.  This family contains
a neighbourhood of the large-volume limit.  The hearts $\Ascr^{\alpha,\beta,H}(Y)$ are obtained from
$\coh(Y)$ by two successive tilts,  first with respect to $\mu_H$-slope stability and then with respect to the tilt-slope function,  both of which are defined from $H$ and the Chern character.  Since $\Phi(H)=H$,  the equivalence $\Phi$ preserves both slope functions and the two torsion pairs, hence it preserves the double-tilt heart
$\Ascr^{\alpha,\beta,H}(Y)$.  The central charge
$Z_{\alpha,\beta,H}^{a,b}$ is expressed entirely in terms of intersections of the Chern character with powers of $H$,  so it is also preserved by $\Phi$.  Therefore $\Phi(\sigma)=\sigma$ for every $\sigma$ in this family.  Applying Lemma \ref{lem:strategy_quintic} to $Y$, $\Sscr_i$, $\Phi$, and $\sigma$ as defined above,  we conclude: 

\begin{theorem}\label{thm:quintic} Let $Y$ be a symmetric quintic threefold as defined above.  Then the space of numerical locally finite Bridgeland stability conditions $\Stab(Y)$ on $Y$ has infinitely many connected components.  More precisely,  if $\sigma$ is any geometric stability condition in the family constructed in \cite[Theorem 1.3]{LiQuintic}, then $T_{\Oscr_{\ell_1}(-1)}^n(\sigma)$ lie in pairwise distinct connected components of $\Stab(Y)$ for all $n \in \ZZ$. 
\end{theorem}

\begin{remark}\label{rem:disconnected_aut} Distinct connected components of $\Stab(Y)$ obtained in Theorem \ref{thm:quintic} are related by autoequivalences of $\Aut \Dscr^b(Y)$ by construction.  However,  from a symmetric quintic $Y$ as above one can also obtain a smooth projective variety $Z$ such that $\Stab(Z)/\Aut \Dscr^b(Z)$ is disconnected.  Here we sketch the idea.  Embed $Y$ into $\PP^5$ via a hyperplane $\{x_5=0\}$ and consider the blow-up ${\pi: Z := \Bl_Y(\PP^5) \to \PP^5}$.  Then $Z$ is a smooth Fano fivefold,  see \cite[\S1, proof of Lemma~1 in the arXiv version]
{BonaveroCampanaWisniewski}. Extend the involution $\Phi$ by fixing $x_5$ and denote its lift to $Z$ by $\Phi_Z \in\Aut(Z)$ .  Put $H=\pi^*\Oscr_{\PP^5}(1)$,  let $j:E\hookrightarrow Z$ be the exceptional divisor and $q:E\to Y$ the projection. 
For $J :=j_*q^*$ and $\Escr_i :=\Oscr_Z(iH)$,  there is a semiorthogonal decomposition 
\[
\Dscr^b(Z)=
\langle\Escr_0,\ldots,\Escr_5,J\Dscr^b(Y)\rangle
\]
by \cite[Theorem 4.3]{OrlovSOD}. 
Moreover
\[
K_{num}(Z)
\cong
\bigoplus_{i=0}^5\ZZ[\Escr_i]\oplus J_*K_{num}(Y).
\]
Now fix a stability condition
$\sigma_0=(Z_0,\Ascr_0) \in \Stab(Y)$ in Li's family and let
$\sigma_1=(Z_1,\Ascr_1):=T_{\Sscr_1}(\sigma_0)$.  We will extend these two stability conditions to stability conditions $\tau_j=(\widetilde Z_j,\Hscr_j)$, $j=0,1$ on $Z$.  The standard cohomology of both hearts $\Ascr_0$ and $\Ascr_1$ is concentrated in degrees at least $-3$.  Indeed,  for $\Ascr_0$ this follows from the way it is constructed by a double tilt of the standard heart and for $\Ascr_1$ it follows from the triangle defining $T_{\Sscr_1}$.  One can independently shift each $\Escr_i$ in such a way that all $\Ext^{\leq 1}$ vanish in both directions between distinct shifted exceptional objects $\Escr_i[p_i]$, $p_i \in \ZZ$ and between each $\Escr_i[p_i]$ and $J\Ascr_j$, $j=0,1$.  Then the gluing recipe of Collins-Polishchuk \cite[Lemma~2.1]{CollinsPolishchuk} gives two hearts of bounded t-structures
\[
\Hscr_j=
\langle\Escr_0[p_0],\dots,\Escr_5[p_5],J\Ascr_j\rangle
\simeq
\Vect_{\CC}^{6}\times\Ascr_j.
\]
Now we can define the numerical central charges by $
\widetilde Z_j(\Escr_i[p_i])=\sqrt{-1}$ and $\widetilde Z_j(JF)=Z_j(F)$. 
Harder-Narasimhan filtrations  are obtained factor-wise and ordered by phase.  Since the summands of a semistable object have the same phase,  the support property of $\sigma_j$ gives a uniform positive lower bound for non-zero semistable masses.  Then by \cite[Lemma~1.1(3)]{CollinsPolishchuk} the resulting stability
conditions $\tau_j=(\widetilde Z_j,\Hscr_j)$ are locally finite.  Their slicings restrict to the slicings of $\sigma_j$ on $J\Dscr^b(Y)$.

Let $D :=5H-E$. Then $\Pic(Z)=\ZZ H\oplus\ZZ E$ and 
$\Nef(Z)=\RR_{\geq 0} H+\RR_{\geq 0} D$.  One has $h^0(Z,H)=6$ and $h^0(Z,D)=h^0(\PP^5,\Iscr_Y(5))=127$, so
an automorphism of $Z$ cannot exchange the generators $H$ and $D$ of the nef cone,  so it fixes them.  Hence every $f\in\Aut(Z)$ fixes $H$ and $E$ and descends via $\pi$ to an automorphism of $\PP^5$  which preserves $Y$.  For $f\in\Aut(Z)$,  let $f_Y$ be the restriction of this lift to $Y$. Then $f\Escr_i\cong\Escr_i$ and $fJ\cong Jf_Y$.  Every automorphism of $Y$ preserves its hyperplane class and thus fixes $\sigma_0$,  as discussed earlier.  Hence $f(\tau_0)=\tau_0$ for every $f\in\Aut(Z)$. The involution $\Phi_Z$ acts trivially on $K_{num}(Z)$.  On the other hand,  restricting to $J\Dscr^b(Y)$ gives
$\Phi_Z(\tau_1)\neq\tau_1$, since $\Phi(\sigma_1)\neq\sigma_1$.

Since $Z$ is Fano,  by \cite[Theorem~3.1]{BondalOrlovReconstruction} every autoequivalence of $\Dscr^b(Z)$ has the form $\psi(-)\simeq(-\otimes L)\circ f[k]$ with $L \in \Pic(Z)$, $f \in \Aut(Z)$ and $k \in \ZZ$.  Fix some $\psi \in \Aut \Dscr^b(Z)$.  Since $f(\tau_0)=\tau_0$, $\Phi_Z(\tau_0)=\tau_0$ and $\Phi_Z(L)\cong L$, we obtain
\[
\Phi_Z(\psi(\tau_0)) = (-\otimes \Phi_Z L)[k](\tau_0) = (-\otimes L)[k](\tau_0) = \psi(\tau_0).
\]
Now let $\Sigma_j$ be the component containing $\tau_j$, $j=0,1$.
The argument as in Lemma \ref{lem:strategy_quintic}(1)
shows that $\Phi_Z$ fixes the component $\psi(\Sigma_0)$ pointwise for every $\psi \in \Aut \Dscr^b(Z)$.  Since $\Phi_Z(\tau_1)\neq\tau_1$,  we conclude that $
\psi(\Sigma_0)\neq\Sigma_1$ for every $\psi\in\Aut\Dscr^b(Z)$. 
\end{remark} 

\subsection{A very general quintic threefold}\label{subsec:verygeneralquintic}

A general quintic threefold has no non-trivial automorphisms,  so the argument in \S \ref{subsec:symmetricquintic} cannot be applied literally.  The key idea is to use \emph{semi-}linear autoequivalences in place of the usual $\CC$-linear ones.  As established by Sosna \cite{Sosna},  semi-linear autoequivalences also act on the space of Bridgeland stability conditions by homeomorphisms,  in the same way as the usual autoequivalences. We recall this action below. 

Let $Y\subset \PP^4$ be a smooth quintic threefold.  Choose a subfield $K\subset\CC$ containing the coefficients of a defining equation of $Y$ and denote by $Y_K\subset\PP^4_K$ the quintic defined by this equation,  so that $Y=Y_K\times_K\CC$.
Any $\theta\in\Aut(\CC/K)$ defines a $\theta$-semi-linear
equivalence $\Phi_\theta$ of $\Dscr^b(Y)$ by the pullback along
$\id_{Y_K}\times\Spec(\theta)$.  Moreover,  such an equivalence $\Phi_\theta$ induces a homeomorphism of $\Stab(Y)$ by
\[
\Phi_\theta(Z,\Pscr)
=\bigl(Z\circ(\Phi_\theta)_*^{-1},\Phi_\theta(\Pscr)\bigr),
\]
where $(\Phi_\theta)_*$ denotes the induced automorphism of $K_{num}(Y)$, 
see \cite[\S 3, before Proposition~3.8]{Sosna}.  For every $\theta\in\Aut(\CC/K)$,  the semi-linear autoequivalence $\Phi_\theta$ fixes all line bundles $\Oscr_Y(j)$.  So $\Phi_\theta$ acts trivially on $K_{num}(Y)$ because the classes of $\Oscr_Y(j)$ for $0\leq j\leq3$ generate $K_{num}(Y)\otimes\QQ$.  The same argument as in \S \ref{subsec:symmetricquintic} can be carried out to conclude that
$\Phi_\theta$ fixes every geometric stability condition in Li's family \cite[Theorem~1.3]{LiQuintic}. 

Let $\ell\subset Y$ be a line.  For $\theta\in\Aut(\CC/K)$ denote by $\ell^\theta$ the line obtained by applying $\theta$ to the equations of $\ell$.  Then $\Phi_\theta$ sends $\Oscr_\ell(-1)$ to $\Oscr_{\ell^\theta}(-1)$.  The exact same argument as in Lemma \ref{lem:strategy_quintic} then shows the following: 

\begin{lemma}\label{lem:strategy_sm_quintic}
Let $Y=Y_K\times_K\CC$ be a smooth quintic and $\sigma \in \Stab(Y)$ be a stability condition from Li's family.  Suppose there is a $(-1,-1)$-line $\ell\subset Y$ satisfying $\ell\cap\ell^\theta=\varnothing$ for some $\theta\in\Aut(\CC/K)$.  Then $T_{\Oscr_\ell(-1)}^n(\sigma)$ for $n\in\ZZ$ are in pairwise distinct connected components of $\Stab(Y)$.  
\end{lemma}

The assumptions of Lemma \ref{lem:strategy_sm_quintic} are satisfied for a very general quintic threefold $Y$ and any line $\ell \subset Y$. 

\begin{theorem}\label{thm:very_general_quintic}
Let $Y\subset\PP^4$ be a very general quintic threefold and let $\sigma \in \Stab(Y)$ be a stability condition from Li's family \cite[Theorem 1.3]{LiQuintic}.  Then for every line $\ell \subset Y$ the stability conditions $T_{\Oscr_\ell(-1)}^n(\sigma)$ with
$n\in\ZZ$ are in pairwise distinct connected components of $\Stab(Y)$.  In particular, $\Stab(Y)$ has infinitely many connected components. 
\end{theorem}

\begin{proof}
A general quintic threefold $Y$ has $2875$ lines,  see
\cite[Appendix~A]{Katz}.  Moreover, they are all pairwise disjoint $(-1,-1)$-lines \cite[Theorems~2.1 and~4.1]{JohnsenKleiman}.  For a very general $Y$ and every two distinct lines $\ell, \ell'$ on $Y$,  we will construct an automorphism of $\CC$ which preserves the coefficients of the quintic $Y$,  but sends $\ell$ to $\ell'$. 

Denote 
\[
P_{\CC} :=\PP(H^0(\PP^4_{\CC},\Oscr(5))),
\quad
P_{\overline{\QQ}}:= 
\PP(H^0(\PP^4_{\overline{\QQ}},\Oscr(5))),
\]
Take an affine coefficient chart in $P_{\CC}$ where one fixed coefficient is non-zero,  normalise it to $1$ and let $c_1,\dots,c_{125}\in\CC$ be the remaining coefficients of $Y$.   In this chart 
$f(c_1,\dots,c_{125})=0$ defines a proper closed subset for every non-zero $f\in\overline{\QQ}[t_1,\dots,t_{125}]$.  Since there are only countably many such polynomials $f$,  for a very general quintic threefold $Y$ the coefficients $c_i$ are algebraically independent over $\overline{\QQ}$. 

Now let $K :=\overline{\QQ}(c_1,\dots,c_{125})$ and let
$Y_K\subset\PP^4_K$ be the quintic defined by the same equation as $Y$.  

Consider the following incidence variety $\Iscr$ over $\overline{\QQ}$: 
\[
\Iscr=
\{([F],\ell)\in
P_{\overline{\QQ}}\times\Gr(2,5)_{\overline{\QQ}}
\mid F|_\ell=0\}.
\]
The projection $q:\Iscr\to\Gr(2,5)_{\overline{\QQ}}$ is a projective bundle (see \cite[proof of Lemma~1.4]{Katz}, which shows that $\Iscr$ is irreducible),  hence $\Iscr$ is integral.
Over some quintic threefold $Z$,  the fibre of the other projection
$p :\Iscr\to P_{\overline{\QQ}}$ is the
Fano scheme of lines $F_1(Z)$ of $Z$,  i.e.  the closed subscheme of the Grassmannian parametrising lines contained in $Z$.  We claim that the generic fibre of $p$ is a spectrum of a field.  Let $\eta \in P_{\overline{\QQ}}$ be the generic point of $P_{\overline{\QQ}}$. Its residue field is $K = \overline \QQ(P_{\overline{\QQ}}) \cong \overline \QQ(t_1,\dots, t_{125})$.  Since the coefficients of $Y$ are algebraically independent over $\overline \QQ$,  we can identify $Y_K$ with the generic quintic over
$\overline{\QQ}$ via the isomorphism $\overline{\QQ}(t_1,\dots,t_{125}) \cong K$ given by $t_i\mapsto c_i$.   Since a general quintic has $2875$ lines,  $p$ is dominant and generically finite of degree $2875$,  i.e.  $F_1(Y_K) = \Iscr_\eta \to \Spec K$ is finite \'etale of degree $2875$.  In particular,  $\Iscr_\eta$ is affine and its coordinate algebra is a $K$-algebra of dimension $2875$.  Since $\Iscr$ is integral,  so is $\Iscr_\eta$ (a localisation of a domain is again a domain).  It follows that the coordinate algebra of $\Iscr_\eta$ is a domain, and hence a field, which we denote by $\kappa$.  We have  $F_1(Y_K)=\Spec \kappa$ and $[\kappa:K]=2875$.  The extension $\kappa/K$ is separable since $K$ has characteristic $0$.

Taking the Fano scheme of lines commutes with field extensions, so we have $F_1(Y) = F_1(Y_K) \times_K \CC = \Spec(\kappa \otimes_K \CC)$.  The lines in $Y$ are in bijection with $\Hom_{\CC-\Sch}(\Spec \CC, F_1(Y)) \cong \Hom_{K-\Sch}(\Spec \CC, \Spec \kappa) \cong \Hom_{K-\Alg}(\kappa,\CC)$,  where the first morphism space is in the category of $\CC$-schemes, the second in the category of $K$-schemes and the third in the category of $K$-algebras.  In other words,  lines in $Y$ correspond to the embeddings
$\iota: \kappa \hookrightarrow\CC$ restricting to the identity on $K$.
Since $\kappa/K$ is finite,  every such embedding factors through the algebraic closure $\overline K$ of $K$
inside $\CC$.  Hence every line in $Y$ is defined over $\overline K$.  By \cite[Proposition~7.4]{Milne}, the action
$\iota\mapsto\gamma\circ\iota$ of $\Gal(\overline K/K)$ on these embeddings, and thus on lines in $Y$,  is transitive.  Now let $\ell, \ell' \subset Y$ be two distinct (and hence disjoint) lines and let $\gamma \in\Gal(\overline K/K)$ be an automorphism sending $\ell$ to $\ell'$.  Choose a transcendence basis $B$ of $\CC/\overline K$.  First extend $\gamma$ to $\overline K(B)$ by putting $\gamma(b) = b$ for $b \in B$.  Then apply \cite[Proposition~7.4]{Milne} again, now to the algebraic closure
$\CC$ of $K(B)$ and the intermediate field $\overline{K}(B)$,  to obtain an automorphism
$\theta\in\Aut(\CC/K)$ extending $\gamma$.
In particular, $\ell^\theta=\ell'$.  Now apply Lemma \ref{lem:strategy_sm_quintic} to $\ell$ and $\ell^\theta = \ell'$.
\end{proof}

\section{Infinitely many braid group orbits} 
In this section we construct a smooth projective variety $X$ such that the braid group action on the set of full exceptional collections in $\Dscr^b(X)$ has infinitely many orbits. 

\subsection{Strategy} 
First we recall the following observation made by Chang,  Haiden and Schroll: 

\begin{lemma}[{\cite[Lemma~5.7]{ChangHaidenSchroll}}]
\label{lem:CHS}
Let $\EE=(\Escr_1,\dots,\Escr_n)$ and $\FF=(\Fscr_1,\dots,\Fscr_n)$ be full exceptional
collections in a $\Hom$-finite triangulated category $\Dscr$.  Suppose that $\EE$ and
$\FF$ lie in the same $B_n\ltimes\ZZ^n$-orbit.  If
$\Phi\in\Aut \Dscr$ fixes $\EE$ term-wise up to shifts (i.e. $\Phi(\Escr_i)\cong \Escr_i[k_i]$ for some $k_i\in\ZZ$ and every $i$), then $\Phi$ also fixes $\FF$ term-wise up to shifts. 
\end{lemma}

The main idea behind our construction is contained in Lemma \ref{lem:strategy} below.  

\begin{lemma}\label{lem:strategy} Let $X$ be a smooth projective variety and $\EE := (\Escr_1, \dots, \Escr_n)$ a full exceptional collection in $\Dscr^b(X)$. Suppose there exist 
\begin{enumerate}
\item an autoequivalence $\Phi \in \Aut \Dscr^b(X)$ such that $\Phi(\Escr_i) \cong \Escr_i$ for all $i$, 
\item a spherical object $\Sscr \in \Dscr^b(X)$ and $j \in \{1,\dots,n\}$ such that $T_\Sscr(\Escr_j)$ is not isomorphic to a shift of $T_{\Phi (\Sscr)}(\Escr_j)$.
\end{enumerate}
Then $\EE$ and $T_\Sscr(\EE) := (T_\Sscr(\Escr_1), \dots, T_\Sscr(\Escr_n))$ are not in the same $B_n \ltimes \ZZ^n$-orbit. 
\end{lemma} 

\begin{proof} By Lemma \ref{lem:CHS} it is sufficient to show that $\Phi$ does not fix $T_\Sscr(\EE)$ up to shifts.  Indeed,  by \cite[Lemma 8.21]{HuybrechtsFM} there is a natural isomorphism $\Phi \circ T_\Sscr \cong T_{\Phi (\Sscr)} \circ \Phi$.
Applying both sides to $\EE$ yields 
\[
\Phi(T_\Sscr(\EE)) \cong  T_{\Phi(\Sscr)}(\Phi(\EE)) \cong T_{\Phi(\Sscr)}(\EE), 
\]
but $T_{\Phi(\Sscr)}(\Escr_j)$ is not isomorphic to a shift of $T_\Sscr(\Escr_j)$ by assumption,  so neither is $\Phi(T_\Sscr(\Escr_j))$.  
\end{proof} 

In \S \ref{sec:construction} we construct $X$, $\EE$, $\Phi$ and $\Sscr$ satisfying the conditions of Lemma \ref{lem:strategy}.  In our construction the autoequivalence $\Phi$ is in fact an automorphism of the variety $X$.  Applying Lemma \ref{lem:strategy} we then immediately see that the braid group action on the set of full exceptional collections in $\Dscr^b(X)$ is not transitive (Theorem \ref{thm:twoorb}).  Infinitely many orbits are then obtained via virtually the same method applied to compositions of spherical twists rather than just one spherical twist (Theorem \ref{thm:main}).

\subsection{Construction}\label{sec:construction}
\subsubsection{The variety} 

We start with the following $6$ points on $\PP^2$
\[
p_1 = [1:0:1], \quad p_2 = [1:0:-1], \quad p_3 = [0:1:1],
\]
\[
p_4 = [0:1:-1], \quad p_5 = [1:2:0], \quad p_6 = [1:3:0].
\]
It is easy to check that no three points are collinear.  Assume that there is a conic 
\[
Ax^2 + Bxy + Cy^2 + Dxz + Eyz + Fz^2 = 0
\]
containing all six points.  Substituting $p_1,p_2,p_3$, and $p_4$ yields $D=0$, $A+F=0$, $E=0$, and $C+F=0$.  Combining this with evaluations at $p_5$ and $p_6$ yields $2B-5F = 0$ and $3B-10F =0$, hence $F=B=0$ and all coefficients vanish.  Hence we conclude that $p_i$ are in general position. 

Let $\overline{C} := \{x^2 + y^2 +z^2 = 0\} \subset \PP^2$ be a smooth conic not passing through any of the points $p_i$.  Consider an involution
\[
\overline{\varphi}\in \PGL(3,\CC), \quad \overline{\varphi}([x:y:z]) = [x:y:-z]. 
\]
Note that $\overline{\varphi}$ fixes $\overline{C}$.  Moreover,  
\[
\overline{\varphi}(p_1) = p_2,  \quad \overline{\varphi}(p_3) = p_4,  \quad \overline{\varphi}(p_5) = p_5, \quad  \overline{\varphi}(p_6) = p_6. 
\]

Let $b: S:= \PP^2_{p_1,\dots,p_6} \to \PP^2$ be the blow-up of the six points.  Let $E_1, \dots, E_6$ be the six exceptional curves and $L := b^* \Oscr_{\PP^2}(1)$.  Since $p_i$ are in general position, $S$ is a del Pezzo surface of degree $3$, so $-K_S = 3L - (E_1 + \dots + E_6)$ is very ample and $|-K_S|$ gives an embedding of $S$ into $\PP^3$ as a smooth cubic surface.  Since $\overline{\varphi}$ preserves the set $\{p_1, \dots,p_6\}$, it has a unique lift to an automorphism $\varphi$ of $S$ satisfying $b \circ \varphi = \overline{\varphi} \circ b$.  Since $\varphi$ preserves the canonical class, it induces an automorphism of $\PP(\H^0(S,-K_S)^\vee) \cong \PP^3$, which we will denote by the same letter $\varphi$.  

Let $C := b^{-1}(\overline{C}) \subset S \subset \PP^3$ be the inverse image of the conic $\overline{C}$.  Since $\overline{C}$ contains none of the six points $p_i$,  the curve $C$ is isomorphic to $\overline{C}$. 

Now let $\ell_j \subset S$ be the strict transform of the unique irreducible conic through the five points $\{p_i\}_{i\neq j}$.  Note that $\ell_j$ automatically does not pass through $p_j$ since no conic contains all six points. 

\begin{lemma}\label{lem:before_bl} \begin{enumerate}
\item $C \cong \PP^1$ is a sextic curve in $\PP^3$ with $[C] = 2L \in \Pic(S)$ and $\varphi(C) = C$. 
\item Each $\ell_j$ is a $(-1)$-curve mapped to a line under the embedding of $S$ into $\PP^3$.  Moreover 
\[
[\ell_j] = 2L - \sum_{i\neq j} E_i \in \Pic(S)
\]
and $C \cdot \ell_j = 4$. 
\item $\varphi(\ell_1)=\ell_2$, $\varphi(\ell_3)=\ell_4$, $\varphi(\ell_5)=\ell_5$, and $\varphi(\ell_6)=\ell_6$.
\end{enumerate}
\end{lemma} 
\begin{proof}
1) Since $[C] = 2L \in \Pic(S)$, one has 
\[
\deg C = (-K_S) \cdot C = (3L - \sum_{i=1}^6 E_i) \cdot 2L = 6.
\]
The invariance of $C$ under $\varphi$ follows from the invariance of $\overline{C}$ under $\overline{\varphi}$.

2) The formula for $[\ell_j]$ follows directly from the definition of $\ell_j$.  Then one has 
\[
\ell_j^2 = (2L)^2 + \sum_{i \neq j} E_i^2 = 4-5 = -1,
\]
\[
C \cdot \ell_j = 2L \cdot (2L - \sum_{i\neq j} E_i) = 4.
\]
The degree of $\ell_j$ in $\PP^3$ is
\[
(-K_S) \cdot \ell_j = (3L - \sum_{i=1}^6 E_i) \cdot (2L - \sum_{i \neq j} E_i) = 6-5 = 1, 
\]
so $\ell_j$ is a line.

3) The automorphism $\overline{\varphi}$ maps the unique conic through $\{p_i\}_{i\neq 1}$ to the unique conic through $\{p_i\}_{i\neq 2}$. Taking strict transforms yields $\varphi(\ell_1)=\ell_2$. The same argument gives $\varphi(\ell_3)=\ell_4$.   Since $\overline{\varphi}$ fixes $p_5$ and $p_6$, it preserves the unique conics through the five points different from $p_5$ and $p_6$, respectively.  Hence $\varphi(\ell_5)=\ell_5$ and $\varphi(\ell_6)=\ell_6$.
\end{proof}

Now let 
\[
\pi: X := \Bl_C \PP^3 \to \PP^3.
\]
Hence $X$ is a smooth projective threefold.  

\subsubsection{The exceptional collection and the spherical objects} Let $j: E \hookrightarrow X$ be the closed embedding of the exceptional divisor and let $p: E \to C$ be the projection.  Put $H = \pi^* \Oscr_{\PP^3}(1)$ and let 
\[
\Psi: \Dscr^b(C) \to \Dscr^b(X) 
\]
\[
\Psi(-) := j_*(p^*(-) \otimes \Oscr_E(-1))
\]
be the Fourier-Mukai functor with kernel $\Oscr_E(-1) = \Oscr_X(E)|_E$.   We write $\Oscr_C(1)$ for a degree-one line bundle on $C$ under the identification $C \cong \PP^1$. 

Orlov's blow-up formula \cite[Theorem 4.3]{OrlovSOD} gives a semiorthogonal decomposition 
\[
\Dscr^b(X) = \langle \Psi(\Dscr^b(C)),  \pi^* \Dscr^b(\PP^3) \rangle
\]
and hence a full exceptional collection 
\begin{equation}\label{eq:FEC}
\EE := (\Psi(\Oscr_C), \Psi(\Oscr_C(1)), \Oscr_X, \Oscr_X(H), \Oscr_X(2H),\Oscr_X(3H))
\end{equation}
obtained by combining the standard full exceptional collections \cite{Beilinson} on $\PP^3$ and $C \cong \PP^1$. 

Since $\varphi(C) = C$ by Lemma \ref{lem:before_bl}(1),  the involution $\varphi$ lifts uniquely to $X$.  We denote the lift by $\Phi$. 

\begin{lemma}\label{lem:phi_fixes} 
The involution $\Phi \in \Aut(X)$ fixes every object of $\EE$ up to an isomorphism.  In particular, $\Phi$ acts trivially on the Grothendieck group $\K_0(X)$. 
\end{lemma} 
\begin{proof}  One has 
\[
\Phi \Oscr_X(H) = \Phi \pi^* \Oscr_{\PP^3}(1) \cong \pi^*\varphi \Oscr_{\PP^3}(1) \cong \pi^* \Oscr_{\PP^3}(1) = \Oscr_X(H).
\] 

Taking tensor powers, we see that $\Phi $ fixes $\Oscr_X(mH)$ for every $m \in \ZZ$,  so in particular it fixes the last four objects of $\EE$. 

The involution $\varphi$ restricts to an automorphism $\varphi_C: C \to C$ of $C \cong \PP^1$.  One has $\varphi_C \Oscr_C \cong \Oscr_C$ and $\varphi_C \Oscr_C(1) \cong \Oscr_C(1)$,  since there is only one isomorphism class of degree-one line bundles on $\PP^1$.  Finally,  note that $\Phi$ preserves $E$, commutes with $p$ and preserves the tautological line bundle $\Oscr_E(-1)$.  So for any $B \in \Dscr^b(C)$ one has 
\[
\Phi \Psi(B) \cong j_* (p^* \varphi_C B \otimes \Oscr_E(-1)) = \Psi(\varphi_C B).
\]
Taking $B = \Oscr_C$ and $B =\Oscr_C(1)$ gives the invariance of the first two objects of $\EE$.  Finally,  $\Phi$ acts trivially on the Grothendieck group since the classes of $\EE$ form a basis of $\K_0(X)$. 
\end{proof}

Let $\widetilde{S} \subset X$ and $\widetilde{\ell_j} \subset X$ denote the strict transforms of $S$ and $\ell_j$ respectively.  Since $S$ is a cubic surface and $S$ contains $C$ with multiplicity one,  $\widetilde{S} \in |3H - E|$.  Moreover,
$\widetilde S\cong \Bl_C S\simeq S$ since $C$ is a Cartier divisor on the smooth surface $S$.  Denote $\Sscr_j:= \Oscr_{\widetilde{\ell_j}}(-1)$. 

\begin{corollary}\label{cor:spherical} For every $1 \leq j \leq 6$, every line bundle on $\widetilde{\ell_j}$ is a $3$-spherical object in $\Dscr^b(X)$.  In particular,  every $\Sscr_j$ is spherical.  Moreover,  
\[
\Phi(\Sscr_1)\cong\Sscr_2,
\quad
\Phi(\Sscr_3)\cong\Sscr_4,
\quad
\Phi(\Sscr_5)\cong\Sscr_5,
\quad
\Phi(\Sscr_6)\cong\Sscr_6.
\] \end{corollary} 

\begin{proof} Fix $j$ and let $\ell := \ell_j$, $\widetilde{\ell} := \widetilde{\ell_j}$.  The scheme-theoretic intersection $\ell\cap C$ is an effective Cartier
divisor (of length $4$,  see Lemma \ref{lem:before_bl}) on the smooth curve $\ell\simeq\PP^1$.  So
the restriction $\pi|_{\widetilde\ell}: \widetilde{\ell} \to \ell$
is an isomorphism and $\widetilde{\ell} \cong \PP^1$.

Now consider the short exact sequence 
\begin{equation}\label{eq:norm}
0 \to N_{\widetilde{\ell}/\widetilde{S}} \to N_{\widetilde{\ell}/X} \to N_{\widetilde{S}/X}|_{\widetilde{\ell}} \to 0.
\end{equation}
Under the isomorphism $\widetilde{S} \cong S$ the curve $\widetilde{\ell}$ maps to the $(-1)$-curve $\ell \subset S$ (Lemma \ref{lem:before_bl}),  thus $N_{\widetilde{\ell}/\widetilde{S}} \cong \Oscr_{\widetilde{\ell}}(\widetilde{\ell}) \cong \Oscr_{\PP^1}(-1)$.  Since $N_{\widetilde{S}/X} \cong \Oscr_X(\widetilde{S})|_{\widetilde{S}}$, one has 
\[
\deg N_{\widetilde{S}/X}|_{\widetilde{\ell}}  = \widetilde{S} \cdot \widetilde{\ell} = (3H - E) \cdot  \widetilde{\ell} = 3-4 = -1.
\]
Here we have used $H \cdot \widetilde{\ell}  = 1$ and $E \cdot  \widetilde{\ell} = \length(C \cap \ell)= C \cdot_S \ell = 4$, because $\Oscr_X(E)|_{\widetilde{\ell}} \cong \Oscr_\ell(C \cap \ell)$.
Since $\widetilde{\ell} \cong \PP^1$, this gives $N_{\widetilde{S}/X}|_{\widetilde{\ell}} \cong \Oscr_{\PP^1}(-1)$.  Then the sequence \eqref{eq:norm} splits since 
\[
\Ext^1_{\PP^1}(\Oscr(-1),\Oscr(-1)) = \H^1(\PP^1,\Oscr) = 0.
\]
Therefore we get $N_{\widetilde{\ell}/X} \cong \Oscr_{\PP^1}(-1) \oplus \Oscr_{\PP^1}(-1)$.  Now apply Lemma \ref{lem:general_sph}. 

The isomorphism $\Phi (\Sscr_1) \cong \Sscr_2$ follows from $\Phi \Oscr_{\widetilde{\ell_1}} \cong \Oscr_{\Phi^{-1}(\widetilde{\ell_1})}$ and the fact that $\Phi$ is an involution, hence $\Phi^{-1}(\widetilde{\ell_1})=\widetilde{\ell_2}$.  The other three isomorphisms follow in the same way from Lemma \ref{lem:before_bl}(3). 
\end{proof}

\begin{lemma}\label{lem:K0same}
The six spherical objects $\Sscr_1,\dots,\Sscr_6$ have the same class in $K_0(X)$.  Consequently, the spherical twists $T_{\Sscr_1},\dots,T_{\Sscr_6}$ induce the same
automorphism of $K_0(X)$.
\end{lemma}

\begin{proof}
Let $\iota_j\colon\widetilde{\ell_j}\hookrightarrow X$ denote the
inclusion.  Since $N_{\widetilde{\ell_j}/X}\cong\Oscr_{\PP^1}(-1)^{\oplus2}$ one has
\[
\ch(\Sscr_j) =
(\iota_j)_*
\left(
\ch\Oscr_{\widetilde{\ell_j}}(-1)
\td (N_{\widetilde{\ell_j}/X})^{-1}
\right) =
[\widetilde{\ell_j}]
\]
 by Grothendieck-Riemann-Roch.  Moreover, $H\cdot\widetilde{\ell_j}=1$ and $E\cdot\widetilde{\ell_j}=4$ for all $j$.  Since $H$ and $E$ span $\Pic(X)$,  the six curves $\widetilde{\ell_j}$ have the same numerical class. Therefore $[\Sscr_1]=\cdots=[\Sscr_6]$ in $K_{num}(X) \cong \K_0(X)$.
\end{proof}

\begin{remark} One has $-K_X \cdot \widetilde{\ell_j} = 0$,  so in particular $X$ is not Fano.  It is in fact weak Fano and appears in Blanc-Lamy's classification of weak Fano blow-ups of curves contained in cubic surfaces \cite[Proposition 4.2(ii)]{BlancLamy},  precisely as the case $(g,d,k;m_1,\dots,m_6) = (0,6,2;0,\dots,0)$.  \end{remark} 

\begin{lemma}\label{lem:two_spherical}
For every $j\in\{1,\dots,6\}$, the object $T_{\Sscr_j}(\Oscr_X(H))$ has exactly two non-zero cohomology sheaves, namely
\[
\Hscr^0\bigl(T_{\Sscr_j}(\Oscr_X(H))\bigr) \cong \Oscr_X(H),
\text { and }
\Hscr^1\bigl(T_{\Sscr_j}(\Oscr_X(H))\bigr) \cong \Sscr_j.
\]
In particular, if $i\neq j$, then $T_{\Sscr_i}(\Oscr_X(H))$ is not isomorphic to a shift of $T_{\Sscr_j}(\Oscr_X(H))$.
\end{lemma}
\begin{proof}
Fix $j\in\{1,\dots,6\}$. Since $\Sscr_j$ is spherical, Serre duality gives
\[
\Ext^k_X(\Sscr_j, \Oscr_X(H)) \cong \Ext_X^{3-k}(\Oscr_X(H), \Sscr_j)^\vee \cong \H^{3-k}(\PP^1, \Oscr_{\PP^1}(-2))^\vee.
\]

For the last isomorphism, we have used 
\[
\RHom_X(\Oscr_X(H),\Sscr_j) \cong \R \Gamma(\widetilde{\ell_j}, \Oscr_X(-H)|_{\widetilde{\ell_j}} \otimes \Oscr_{\widetilde{\ell_j}}(-1))
\] by the projection formula and $\Oscr_X(H)|_{\widetilde{\ell_j}} \cong \Oscr_{\PP^1}(1)$, so $\Oscr_X(-H)|_{\widetilde{\ell_j}} \otimes \Oscr_{\widetilde{\ell_j}}(-1) \cong \Oscr_{\PP^1}(-1) \otimes \Oscr_{\PP^1}(-1) \cong \Oscr_{\PP^1}(-2)$. 

Thus $\Ext^k_X(\Sscr_j, \Oscr_X(H))$ is one-dimensional for $k=2$ and vanishes otherwise. Hence the triangle defining $T_{\Sscr_j}(\Oscr_X(H))$ is
\[
\Sscr_j[-2] \to \Oscr_X(H) \to T_{\Sscr_j}(\Oscr_X(H)) \to \Sscr_j[-1],
\]
and the long exact sequence of cohomology yields the claim. Note that $\Sscr_i\cong\Sscr_j$ for $i \neq j$ is impossible because these sheaves have different supports.
\end{proof} 

\subsection{Infinitely many braid group orbits}
We denote the autoequivalence of $\Dscr^b(X)$ induced by $\Phi$ by the same letter. 

In order to illustrate the general idea in the simplest possible form,  first we observe that there are at least two orbits,  which already gives a smooth projective counterexample to \cite[Conjecture 2.2]{BondalPolishchuk}: 

\begin{theorem}\label{thm:twoorb} Let $X = \Bl_C \PP^3$ be as constructed above,  let $\EE$ be the exceptional collection \eqref{eq:FEC} and $\Sscr_1 = \Oscr_{\widetilde{\ell_1}}(-1)$.  Then the exceptional collections $\EE$ and $T_{\Sscr_1}(\EE)$ are not in the same $B_6 \ltimes \ZZ^6$-orbit.  In particular,  the action of $B_6 \ltimes \ZZ^6$ by mutations and shifts on the set of full exceptional collections is not transitive. \end{theorem}  

\begin{proof} 
By Lemma \ref{lem:phi_fixes} $\Phi$ fixes every object of $\EE$.  By Lemma \ref{lem:two_spherical} $T_{\Sscr_1}(\Oscr_X(H))$ is not isomorphic to a shift of $T_{\Sscr_2}(\Oscr_X(H))$.  By Corollary \ref{cor:spherical} $\Sscr_2 \cong \Phi(\Sscr_1)$.  Applying Lemma \ref{lem:strategy} to $X$, $\EE$, $\Phi$, $\Sscr = \Sscr_1$ and $j = 4$ finishes the proof. 
\end{proof} 

Now we will refine this result, proving in particular that there are in fact infinitely many $B_6 \ltimes \ZZ^6$-orbits. 

For $j\in\{1,\dots,6\}$ denote $T_j:=T_{\Sscr_j}$. 

If $i\neq j$, then
\[
\ell_i\cdot \ell_j
=
(2L-\sum_{k\neq i}E_k)\cdot
(2L-\sum_{k\neq j}E_k)
=4-4=0.
\]
So $\ell_i$ and $\ell_j$ are disjoint and hence their strict transforms $\widetilde{\ell_i}$ and $\widetilde{\ell_j}$ are also disjoint.  Thus $\RHom_X(\Sscr_i,\Sscr_j)=0$ for $i \neq j$,  so the twists $T_1,\dots,T_6$ commute (see \cite[Proposition 2.12]{SeidelThomas}).  

For $\aa=(a_1,a_2,a_3,a_4,a_5,a_6)\in\ZZ^6$ we put
\[
\Tscr_{\aa}:=T_1^{a_1}T_2^{a_2}T_3^{a_3}T_4^{a_4}T_5^{a_5}T_6^{a_6}
\]
and let $\EE_{\aa}$ be the full exceptional collection obtained by applying $\Tscr_{\aa}$ to $\EE$: 
\[
\EE_{\aa}:=\Tscr_{\aa}(\EE).
\]

By Corollary~\ref{cor:spherical},
\begin{equation}\label{eq:conj_twists}
\begin{aligned}
\Phi T_1\Phi^{-1}&\cong T_2,
&\Phi T_2\Phi^{-1}&\cong T_1,
&\Phi T_3\Phi^{-1}&\cong T_4,\\
\Phi T_4\Phi^{-1}&\cong T_3,
&\Phi T_5\Phi^{-1}&\cong T_5,
&\Phi T_6\Phi^{-1}&\cong T_6.
\end{aligned}
\end{equation}

For $(r,s)\in\ZZ^2$,  we write $\underline{(r,s)}$ for the tuple $(r,-r,s,-s,0,0) \in \ZZ^6$,  so
\[
\Tscr_{\underline{(r,s)}} =\Tscr_{(r,-r,s,-s,0,0)} = T_1^rT_2^{-r}T_3^sT_4^{-s}.
\]

\begin{lemma}\label{lem:detect}
Let $a,b,k\in\ZZ$.  If $\Tscr_{\underline{(a,b)}}(\Oscr_X(H))\cong\Oscr_X(H)[k]$,  then $a=b=k=0$.
\end{lemma}

\begin{proof}
Since the $\Sscr_j$ are $3$-spherical,  one has $T_j(\Sscr_j)\cong\Sscr_j[-2]$.  Since the twists $T_j$ commute,  we get 
\[
\Tscr_{\underline{(a,b)}}^{-1}(\Sscr_1)\cong\Sscr_1[2a],
\quad
\Tscr_{\underline{(a,b)}}^{-1}(\Sscr_2)\cong\Sscr_2[-2a].
\]
As shown in the proof of Lemma \ref{lem:two_spherical},
\[
\RHom_X(\Sscr_j,\Oscr_X(H))\cong\CC[-2].
\]
Therefore
\[
\RHom_X(\Sscr_1,\Tscr_{\underline{(a,b)}}(\Oscr_X(H)))\cong  \RHom_X(\Tscr_{\underline{(a,b)}}^{-1}(\Sscr_1),\Oscr_X(H)) \cong \CC[-2-2a], 
\]
\[
\RHom_X(\Sscr_2,\Tscr_{\underline{(a,b)}}(\Oscr_X(H))) \cong \RHom_X(\Tscr_{\underline{(a,b)}}^{-1}(\Sscr_2), \Oscr_X(H)) \cong \CC[-2+2a].
\]
Assume $\Tscr_{\underline{(a,b)}}(\Oscr_X(H))\cong\Oscr_X(H)[k]$.  Then both complexes above are also isomorphic to $\CC[-2+k]$.  Thus $k=-2a=2a$, so $a=k=0$.  The same argument applied to $\Sscr_3$ and $\Sscr_4$ gives $b=0$.
\end{proof}

\begin{theorem}\label{thm:main}
Let $\aa=(a_1,\dots,a_6)$,  $\bb=(b_1,\dots,b_6) \in \ZZ^6$.  If the full exceptional collections $\EE_{\aa}$ and $\EE_{\bb}$ lie in the same $B_6\ltimes\ZZ^6$-orbit,  then
\[
a_1-a_2=b_1-b_2 \text{ and }
a_3-a_4=b_3-b_4.
\]
\end{theorem}

\begin{proof}
Since $\Phi$ fixes every object of $\EE$ by Lemma \ref{lem:phi_fixes}, the autoequivalence $\Tscr_{\aa}\Phi\Tscr_{\aa}^{-1}$ fixes every object of $\EE_{\aa}$.  Suppose that $\EE_{\aa}$ and $\EE_{\bb}$ are in the same $B_6\ltimes\ZZ^6$-orbit.  Then by \cite[Lemma~5.7]{ChangHaidenSchroll},  the same autoequivalence $\Tscr_{\aa}\Phi\Tscr_{\aa}^{-1}$ fixes every object of $\EE_{\bb}$ up to shifts.  Applying $\Tscr_{\bb}^{-1}$ to the resulting isomorphisms, we see that  $\Tscr_{\aa-\bb}\Phi\Tscr_{\aa-\bb}^{-1}$
fixes every object of $\EE$ up to a shift.  Since $\Phi^{-1}$ fixes $\EE$ term-wise, so does
\[
\Tscr_{\aa-\bb}\Phi\Tscr_{\aa-\bb}^{-1}\Phi^{-1}.
\]
Put $d_1 :=(a_1-a_2)-(b_1-b_2)$, $d_2 :=(a_3-a_4)-(b_3-b_4)$. 
Using \eqref{eq:conj_twists} and the commutativity of the twists $T_i$,  we get
\[
\Tscr_{\aa-\bb}\Phi\Tscr_{\aa-\bb}^{-1}\Phi^{-1}
\cong
\Tscr_{\underline{(d_1,d_2)}}.
\]
In particular,
\[
\Tscr_{\underline{(d_1,d_2)}}(\Oscr_X(H))\cong\Oscr_X(H)[k]
\]
for some $k\in\ZZ$.  Thus by Lemma \ref{lem:detect} we conclude $d_1=d_2=0$. 
\end{proof}

\begin{corollary}\label{cor:inf}
For $(r,s)\in\ZZ^2$ the full exceptional collections $T_1^rT_3^s(\EE)= \EE_{(r,0,s,0,0,0)}$ lie in pairwise distinct $B_6\ltimes\ZZ^6$-orbits. In particular,  the action of $B_6\ltimes\ZZ^6$ on the set of full exceptional collections in $\Dscr^b(X)$ has infinitely many orbits.
\end{corollary}

The following corollary shows that there is,  moreover,  an infinite family of orbits which are indistinguishable on the level of $K_0$.  Denote $\EE_{\underline{(r,s)}} =\Tscr_{\underline{(r,s)}}(\EE)$. 

\begin{corollary}\label{cor:K0_invisible}
For $(r,s)\in\ZZ^2$ the full exceptional collections $\EE_{\underline{(r,s)}}$ lie in pairwise distinct $B_6\ltimes\ZZ^6$-orbits.  Moreover,  every autoequivalence $\Tscr_{\underline{(r,s)}}$ acts trivially on $K_0(X)$.  In particular, $\EE_{\underline{(r,s)}}$ and $\EE$ coincide in $K_0(X)$ term-wise.
\end{corollary}

\begin{proof}
The fact that the collections $\EE_{\underline{(r,s)}}$ are in pairwise distinct orbits is immediate from Theorem \ref{thm:main}. Looking at the formula defining $\Tscr_{\underline{(r,s)}}$ and applying Lemma \ref{lem:K0same} we conclude that it acts trivially on $K_0(X)$.
 \end{proof}

\begin{remark} At the moment we do not attempt to classify all $B_6\ltimes\ZZ^6$-orbits.  We note,  however,  that the collections $\EE_{\aa}$ with $\aa \in \ZZ^6$ considered above do not represent \emph{all} orbits.  More orbits can be produced by considering other line bundles on the curves $\widetilde{\ell_i}$ (which are spherical by Corollary \ref{cor:spherical}).  

For instance,  for $1 \leq i \leq 4$ the full exceptional collection $T_{\Oscr_{\widetilde{\ell_i}}}(\EE)$ does not lie in the same $B_6\ltimes\ZZ^6$-orbit as $\EE_{\aa}$ for any $\aa \in\ZZ^6$.  The argument is similar to the proof of Theorem \ref{thm:main},  so we will only sketch it.  Let $i=1$ (the other cases are identical).  Denote $U_j := T_{\Oscr_{\widetilde{\ell_j}}}$.  Suppose for a contradiction that $U_1(\EE)$ and $\EE_{\aa}$ are in the same $B_6\ltimes\ZZ^6$-orbit.  Since $U_1\Phi U_1^{-1}$ fixes every object of $U_1(\EE)$, Lemma~\ref{lem:CHS} implies that $U_1\Phi U_1^{-1}$ fixes $\EE_{\aa}$ term-wise up to shifts.  Conjugating by $\Tscr_{\aa}^{-1}$ and composing with $\Phi^{-1}$,
we conclude that $Q := \Tscr_{\aa}^{-1}U_1\Phi U_1^{-1}\Tscr_{\aa}\Phi^{-1}$ fixes $\EE$ term-wise up to shifts.  Using $\Phi U_1\Phi^{-1}\cong U_2$ we can rewrite $Q$ as $Q = \Tscr_{\aa}^{-1}U_1U_2^{-1}(\Phi\Tscr_{\aa}\Phi^{-1})$.  We then apply $Q$ to $\Oscr_X$,  which is the third object of $\EE$.  It is easy to see that the twists $T_i$ fix $\Oscr_X$,  so $Q(\Oscr_X) \cong \Oscr_X[k]$ implies $\Tscr_{\aa}^{-1}U_1U_2^{-1}(\Oscr_X) \cong  \Oscr_X[k]$.  Applying $\RHom_X(\Sscr_1,-)$ to both sides yields a contradiction since $\RHom_X(\Sscr_1,\Oscr_X[k]) = 0$ for every $k \in \ZZ$, while 
\[
\begin{aligned}
\RHom_X(\Sscr_1,\Tscr_{\aa}^{-1}U_1U_2^{-1}(\Oscr_X))
&\cong
\RHom_X(\Tscr_{\aa}(\Sscr_1),U_1U_2^{-1}(\Oscr_X))\\
&\cong
\RHom_X(\Sscr_1,U_1(\Oscr_X))[2a_1],
\end{aligned}
\]
which is easily seen to be non-zero. 
\end{remark}

Now put
\begin{equation}\label{eq:Phirs}
\Phi_{r,s}:=\Tscr_{\underline{(r,s)}}\Phi\Tscr_{\underline{(r,s)}}^{-1}.
\end{equation}

We record the following,  which will be used in the next section.

\begin{lemma}\label{lem:cross_stab}
Let $(r,s),(r',s')\in\ZZ^2$. Then $\Phi_{r,s}$ fixes $\EE_{\underline{(r',s')}}$ term-wise up to shifts if and only if $(r,s)=(r',s')$.
\end{lemma}

\begin{proof}
The ``if'' direction is immediate.  Now suppose that $\Phi_{r,s}$ fixes every term of $\EE_{\underline{(r',s')}}$ up to shifts.  Conjugating by $\Tscr_{\underline{(r',s')}}^{-1}$ shows that $
\Tscr_{\underline{(r-r',s-s')}}\Phi\Tscr_{\underline{(r-r',s-s')}}^{-1}$ fixes $\EE$ term-wise up to shifts.  Since $\Phi^{-1}$ fixes $\EE$ term-wise,  so does
\[
Q := \Tscr_{\underline{(r-r',s-s')}}\Phi\Tscr_{\underline{(r-r',s-s')}}^{-1}\Phi^{-1}.
\]
By \eqref{eq:conj_twists}
\[
\Phi\Tscr_{\underline{(r-r',s-s')}}^{-1}\Phi^{-1}
\cong
\Tscr_{\underline{(r-r',s-s')}}
\]
and hence $Q \cong \Tscr_{\underline{(2(r-r'),2(s-s'))}}$.
Thus
\[
\Tscr_{\underline{(2(r-r'),2(s-s'))}}(\Oscr_X(H))\cong\Oscr_X(H)[k]
\]
for some $k\in\ZZ$.  Lemma \ref{lem:detect} yields $(r,s)=(r',s')$.
\end{proof}

\subsection{Infinitely many connected components of $\Stab(X)$}\label{sec:stab}
Recall that by $\Stab(X)$ we denote the space of locally finite (numerical) Bridgeland stability conditions on $\Dscr^b(X)$.  It is a finite-dimensional complex manifold equipped with a natural left action of $\Aut \Dscr^b(X)$ and a commuting right action of $\widetilde{\GL}^+(2,\RR)$.  We refer to \cite{BridgelandStab} for details.  

Let $\FF=(\Fscr_1,\dots,\Fscr_6)$ be a full exceptional collection in $\Dscr^b(X)$.  Following Macr\`i's \cite[\S 3.3]{MacriStab} we define an open subset $\Theta_{\FF} \subset \Stab(X)$ in the following way.  Let $p_1,\dots,p_6 \in \ZZ$ be such that $\FF_{\ppp} := (\Fscr_1[p_1],\dots,\Fscr_6[p_6])$ satisfies $\Ext^{\leq 0}(\Fscr_i[p_i],\Fscr_j[p_j]) = 0$ for all $i\neq j$. 
Let 
\[
\Ascr_{\FF,\ppp}
:=
\langle \Fscr_1[p_1],\dots,\Fscr_6[p_6]\rangle_{ext}
\]
be the extension-closed subcategory generated by $\Fscr_1[p_1],\dots,\Fscr_6[p_6]$.  By \cite[Lemma 3.14]{MacriStab} $\Ascr_{\FF,\ppp}$ is the heart of a bounded $t$-structure with simple objects  $\Fscr_1[p_1],\dots,\Fscr_6[p_6]$.  For arbitrary $z_1,\dots,z_6\in\HH =\{r\exp(i\pi\phi):r>0,\ 0<\phi\leq 1\}$,
setting 
\[
Z(\Fscr_i[p_i])=z_i, \quad i=1,\dots,6
\]
defines a unique locally finite stability condition with heart $\Ascr_{\FF,\ppp}$ (see \cite[Remark 2.2,  \S 3.3]{MacriStab}).  Let $\Theta_{\FF}$ be the set of stability conditions obtained in this way by varying $\ppp \in \ZZ^6$ and $z_i \in \HH$, up to the action of $\widetilde{\GL}^+(2,\RR)$.  By \cite[Lemma 3.19]{MacriStab} $\Theta_{\FF}$ is an open,  connected and simply connected submanifold of $\Stab(X)$.

\begin{lemma}\label{lem:stab_components}
Let $\FF$ and $\GG$ be full exceptional collections in $\Dscr^b(X)$, and let $\Phi$ be an autoequivalence which fixes every term of $\FF$.  If $\Phi$ does not fix every term of $\GG$ up to shifts, then $\Theta_{\FF}$ and $\Theta_{\GG}$ lie in different connected components of $\Stab(X)$.
\end{lemma}

\begin{proof}
Since $\FF$ is full,  the classes of its objects form a basis of $K_0(X)$.  Thus $\Phi$ acts trivially on $K_0(X)$.  Moreover, $\Phi$ fixes $\Theta_{\FF}$ pointwise.  Indeed, it preserves each heart $\Ascr_{\FF,\ppp}$,  fixes its simple objects and does not change the central charge,  since the actions $\Aut \Dscr^b(X)$ and $\widetilde{\GL}^+(2,\RR)$ commute.  Since $\Theta_{\FF}$ is open, the identity theorem implies that $\Phi$ acts trivially on the connected component of $\Stab(X)$ containing $\Theta_{\FF}$.

Assume that $\Theta_{\GG}$ lies in the same component.  Choose  $q_1,\dots,q_6 \in \ZZ$ such that the shifted collection $\GG_{\qqq} = (\Gscr_1[q_1],\dots,\Gscr_6[q_6])$ has no non-positive $\Ext$'s.  Let
$\sigma\in\Theta_{\GG}$ be a stability condition whose heart is $\Ascr_{\GG,\qqq}$.  Since $\Phi$ acts trivially on the whole connected component containing $\sigma$,  it in particular fixes $\sigma$,  hence induces an autoequivalence of the finite-length abelian category $\Ascr_{\GG,\qqq}$.  It therefore maps each simple object to a simple object.  Moreover,  since $\Phi$ acts trivially on $\K_0(X)\cong \K_0(\Ascr_{\GG,\qqq})$, we have 
\[
[\Phi(\Gscr_i[q_i])]=[\Gscr_i[q_i]]
\]
in $K_0(\Ascr_{\GG,\qqq})$.  But the classes of the simple objects form a basis of $K_0(\Ascr_{\GG,\qqq})$, hence $
\Phi(\Gscr_i[q_i])\cong \Gscr_i[q_i]$ for every $i$,  contrary to the assumption. 
\end{proof}

\begin{theorem}\label{thm:stab}
The map
\[
\ZZ^2\longrightarrow\pi_0(\Stab(X)),
\quad
(r,s)\longmapsto
\text{the connected component of }\Theta_{\EE_{\underline{(r,s)}}},
\]
is injective.  In particular,  $\Stab(X)$ has infinitely many connected components.  Moreover, there exists
\[
Z\in\Hom_{\ZZ}(K_0(X),\CC)
\]
such that the fibre $\Zscr^{-1}(Z)$ meets infinitely many connected components of $\Stab(X)$. 
\end{theorem}

\begin{proof}
The autoequivalence $\Phi_{r,s}$ (see \eqref{eq:Phirs}) fixes every term of $\EE_{\underline{(r,s)}}$.  By Lemma \ref{lem:cross_stab} for any $(r',s') \neq (r,s)$ the autoequivalence $\Phi_{r,s}$ does not fix $\EE_{\underline{(r',s')}}$ term-wise up to shifts.  Lemma \ref{lem:stab_components} thus implies that $\Theta_{\EE_{\underline{(r,s)}}}$ and $\Theta_{\EE_{\underline{(r',s')}}}$ lie in distinct connected components.

Now choose $\sigma=(Z,\Pscr)\in\Theta_{\EE}$.  By Corollary \ref{cor:K0_invisible} the autoequivalence $\Tscr_{\underline{(r,s)}}$ acts trivially on $\K_0(X)$.  Therefore the central charge of $\Tscr_{\underline{(r,s)}}(\sigma)$ is $Z\circ\Tscr_{\underline{(r,s)}}^{-1}=Z$. 
By the definition of $\Theta_{\EE_{\underline{(r,s)}}}$, one has $\Tscr_{\underline{(r,s)}}(\sigma)\in\Theta_{\EE_{\underline{(r,s)}}}$. 
Since these stability conditions lie in pairwise distinct connected components,  it follows that $\Zscr^{-1}(Z)$ meets infinitely many connected components.
\end{proof} 

\bibliography{bibs2}
\bibliographystyle{amsplain}
\end{document}